\documentclass[11pt]{amsart}

\usepackage{amsmath,amssymb,amsthm}
\usepackage{geometry}
\usepackage{booktabs}
\usepackage{array}
\usepackage{hyperref}
\hypersetup{colorlinks=true, linkcolor=blue, citecolor=blue, urlcolor=blue}

\newtheorem{theorem}{Theorem}[section]
\newtheorem{lemma}[theorem]{Lemma}
\newtheorem{proposition}[theorem]{Proposition}
\newtheorem{corollary}[theorem]{Corollary}
\newtheorem{conjecture}[theorem]{Conjecture}
\newtheorem{definition}[theorem]{Definition}
\newtheorem{example}[theorem]{Example}
\newtheorem{remark}[theorem]{Remark}

\newcommand{\LE}[1]{\left\langle\!\!\left\langle #1 \right\rangle\!\!\right\rangle}
\newcommand{\EN}[1]{\left\langle #1 \right\rangle}

\title[Latin Eulerian Numbers]{Latin Eulerian Numbers}

\author{Madjid Mirzavaziri}
\address{\bf Department of Pure Mathematics, Ferdowsi University of Mashhad, Mashhad, Iran}
\email{mirzavaziri@um.ac.ir}
\author{Daniel Yaqubi$^{*}$}
\address{\bf $^*$Department of Computer science, University of Torbat-e Jam, Torbat-e Jam, Iran.}
\email{yaqubi@tjamcaas.ac.ir, or daniel\_yaqubi@yahoo.es}
\thanks{Corresponding author: Daniel Yaqubi}

\begin{document}
\begin{abstract}
We introduce \emph{Latin Eulerian numbers} $\LE{n\atop k_1,\dots,k_n}$, a multivariate refinement of classical Eulerian numbers counting order-$n$ Latin squares by column ascents. We establish their fundamental symmetries, univariate marginals, and an exact multiplicity divisibility property. For the total ascent statistic $\Sigma(L)=\sum_i k_i(L)$, we prove the sharp, isolated bounds $n-1 \le \Sigma(L) \le (n-1)^2$, demonstrating that the adjacent values $n$ and $(n-1)^2-1$ are strictly unattainable. To analyze intermediate values, we show that symbol permutations typically ignored in classical enumeration directly govern $\Sigma$ via an exact formula. This insight yields explicit constructions for the interior range, motivates a unimodality conjecture for the total ascent distribution.
\end{abstract}

\maketitle

\section{Introduction}
The \textit{Eulerian numbers} $\EN{n\atop k}$, which count permutations $\pi$ of $[n]=\{1,\dots,n\}$ with exactly $k$ ascents (positions $i$ with $\pi(i)<\pi(i+1)$), are among the most studied objects in enumerative combinatorics (see \cite{Comtet, Riordan, StanleyEC1, Petersen}). Over the years, numerous refinements and generalizations of $\EN{n\atop k}$ have been introduced, including the $q$-Eulerian numbers of Carlitz, the Eulerian numbers of type $B$, and second-order Eulerian numbers on multiset permutations.

Another classical and richly developed domain in combinatorics is the theory of \textit{Latin squares}. An order-$n$ Latin square is an $n\times n$ array with entries in $[n]$ such that every row and every column contains each symbol exactly once. Euler himself initiated the systematic study of orthogonal Latin squares through his $36$ officers problem~\cite{EulerOfficers}; the resulting theory of orthogonal Latin squares, transversal designs, and quasigroups is surveyed in the standard reference of D\'enes and Keedwell~\cite{DenesKeedwell}. The number $L_n$ of Latin squares of order $n$ is known exactly only for $n\le 11$~\cite{McKayWanless}, and no closed formula is known; the best asymptotic estimates come from the permanent-based approach following van der Waerden's conjecture (proved by Egorychev~\cite{Egorychev} and Falikman~\cite{Falikman}) and, more recently, from entropy methods of Linial and Luria~\cite{LinialLuria}.

It is natural to ask whether the two theories can be fruitfully combined. Because every column of a Latin square forms a permutation of $[n]$, one may refine the enumeration of Latin squares by recording the number of ascents in \emph{each} column simultaneously. This produces an $n$-tuple of Eulerian-type statistics attached to every Latin square, and hence a genuinely multivariate array generalizing $\EN{n\atop k}$. We call the resulting numbers \emph{Latin Eulerian numbers}. Besides Euler's own connection to both source concepts, the name is intended to signal that this array specializes, in a suitable sense, to $n$ independent copies of the classical Eulerian statistic, subject to the rigid combinatorial constraint of Latin-ness.

The purpose of this note is fourfold: (i) to define Latin Eulerian numbers and record their basic symmetries; (ii) to compute the array explicitly for $n \le 5$; (iii) to establish sharp, isolated bounds for the total ascent statistic $\Sigma = \sum_i k_i$ (Theorems~\ref{thm:bounds} and~\ref{thm:isolation}) via a row-pair comparison argument; and (iv) to analyze how classical isotopy symmetries (row, column, and symbol permutations) control intermediate values of $\Sigma$. Specifically, column permutations leave $\Sigma$ invariant (Proposition~\ref{prop:perm}), restricted row permutations yield an explicit number-theoretic construction (\S\ref{sec:interior}), and symbol permutations supply the decisive mechanism to bridge the remaining gap (Theorem~\ref{thm:symbol-general}). Finally, we present constructive evidence for a conjecture on interior coverage and unimodality (Conjecture~\ref{conj:interior}).
\section{Definitions and basic properties}

Throughout, $n\ge 2$ is an integer and $[n]=\{1,2,\dots,n\}$.

\begin{definition}
For a permutation $\pi$ of $[n]$, an \emph{ascent} of $\pi$ is an index $i\in\{1,\dots,n-1\}$ with $\pi(i)<\pi(i+1)$. We write $\mathrm{asc}(\pi)$ for the number of ascents of $\pi$, so $0\le \mathrm{asc}(\pi)\le n-1$ and $\EN{n\atop k}=\#\{\pi\in S_n:\mathrm{asc}(\pi)=k\}$.
\end{definition}

\begin{definition}
An (order-$n$) \emph{Latin square} is a matrix $L=(L(i,j))_{1\le i,j\le n}$ with entries in $[n]$ such that each row and each column of $L$ is a permutation of $[n]$. Let $\mathcal L_n$ denote the set of order-$n$ Latin squares, and $L_n=|\mathcal L_n|$.
\end{definition}

\begin{definition}\label{def:main}
For $L\in\mathcal L_n$ and $1\le i\le n$, let $c_i(L)=(L(1,i),L(2,i),\dots,L(n,i))$ be the $i$-th column of $L$, viewed as a permutation of $[n]$, and set $k_i(L)=\mathrm{asc}(c_i(L))\in\{0,\dots,n-1\}$. For a composition $\mathbf k=(k_1,\dots,k_n)\in\{0,\dots,n-1\}^n$, the \emph{Latin Eulerian number} $\LE{n\atop k_1,\dots,k_n}$ is
\[
\LE{n\atop k_1,\dots,k_n} \;=\; \#\bigl\{\,L\in\mathcal L_n \;:\; k_i(L)=k_i \text{ for all } i=1,\dots,n \,\bigr\}.
\]
\end{definition}

\begin{example}\label{ex:311}
For $n=3$ the Latin square $L=\begin{pmatrix}1&2&3\\3&1&2\\2&3&1\end{pmatrix}$
has columns $(1,3,2)$, $(2,1,3)$, $(3,2,1)$, with respectively $1$, $1$, $0$ ascents. Hence $L$ is one of the squares counted by $\LE{3\atop 1,1,0}$.
\end{example}
Now, we record three elementary but structurally important properties column-permutation symmetry (Proposition~\ref{prop:perm}), row-reversal symmetry (Proposition~\ref{prop:rev}), and total mass (Proposition~\ref{prop:sum}) of Latin Eulerian numbers. 
\begin{proposition}\label{prop:perm}
For every permutation $\tau$ of $[n]$ and every composition $\mathbf k$,
\[
\LE{n\atop k_1,\dots,k_n} \;=\; \LE{n\atop k_{\tau(1)},\dots,k_{\tau(n)}}.
\]
In particular, $\LE{n\atop k_1,\dots,k_n}$ depends only on the multiset $\{k_1,\dots,k_n\}$.
\end{proposition}
\begin{proof}
Permuting the columns of $L$ according to $\tau$ (i.e. replacing $L$ by $L'$ with $L'(\cdot,\tau(i))=L(\cdot,i)$) is a bijection $\mathcal L_n\to\mathcal L_n$, and it carries a square with column-ascent vector $(k_1,\dots,k_n)$ to one with column-ascent vector $(k_{\tau^{-1}(1)},\dots,k_{\tau^{-1}(n)})$; since $\tau$ ranges over all of $S_n$ so does $\tau^{-1}$.
\end{proof}

\begin{proposition}\label{prop:rev}
The Latin Eulerian numbers satisfy the identity
\[
\LE{n\atop k_1,\dots,k_n} = \LE{n\atop n-1-k_1,\dots,n-1-k_n}
\]
for all $0 \le k_1, \dots, k_n \le n-1$.
\end{proposition}
\begin{proof}
Reversing the order of the rows of $L$ (row $i\mapsto$ row $n+1-i$) is a bijection $\mathcal L_n\to\mathcal L_n$ which turns every ascent of every column into a descent and vice versa, replacing $k_i$ by $(n-1)-k_i$.
\end{proof}

\begin{proposition}\label{prop:sum}
The sum of the Latin Eulerian numbers over all possible ascent tuples equals the total number $L_n$ of Latin squares of order $n$, means
\[
\sum_{\mathbf{k}\in\{0,\dots,n-1\}^n} \LE{n\atop k_1,\dots,k_n} = L_n.
\]
\end{proposition}
\begin{proof}
This follows immediately, as every Latin square of order $n$ uniquely determines its ascent tuple $\mathbf{k} \in \{0, \dots, n-1\}^n$.
\end{proof}
Proposition~\ref{prop:sum} is the sense in which Latin Eulerian numbers refine $L_n$, exactly as $\EN{n\atop k}$ refines $n!=\sum_k \EN{n\atop k}$; since no closed form for $L_n$ is known, none should be expected in general for $\LE{n\atop \mathbf k}$ either, and we do not attempt one. We instead study the array computationally for small $n$ and analytically through the derived statistic $\Sigma(L)=\sum_i k_i(L)$.

Our final basic property establishes that the univariate marginals of the Latin Eulerian array coincide with the classical Eulerian numbers, scaled by the factor $L_n / n!$.
\begin{theorem}\label{thm:marginal}
Fix a column index $i\in\{1,\dots,n\}$ and $k\in\{0,\dots,n-1\}$. Then
\[
\sum_{\substack{\mathbf k'\,:\,k'_i=k}} \LE{n\atop \mathbf k'} \;=\; \frac{L_n}{n!}\,\EN{n\atop k}.
\]
\end{theorem}
\begin{proof}
It suffices to show that, among the $L_n$ Latin squares of order $n$, each of the $n!$ permutations of $[n]$ occurs as column $i$ equally often, namely $L_n/n!$ times; summing this over the $\EN{n\atop k}$ permutations with exactly $k$ ascents gives the stated formula.

Fix two permutations $\sigma,\sigma'\in S_n$ and let $\rho=\sigma^{-1}\sigma'$. Given a Latin square $L$ with column $i$ equal to $\sigma$, define $L'$ by $L'(j,c)=L(\rho(j),c)$ for all rows $j$ and columns $c$; this is again a Latin square, since reordering the rows of a Latin square (here, by the permutation $\rho$ of row indices) preserves the Latin property. Its $i$-th column satisfies $L'(j,i)=L(\rho(j),i)=\sigma(\rho(j))=\sigma(\sigma^{-1}\sigma'(j))=\sigma'(j)$, so column $i$ of $L'$ is exactly $\sigma'$. The map $L\mapsto L'$ is invertible (apply the same construction with $\rho^{-1}=\sigma'^{-1}\sigma$), so it is a bijection between $\{L\in\mathcal L_n : \text{column }i\text{ of }L=\sigma\}$ and $\{L\in\mathcal L_n:\text{column }i\text{ of }L=\sigma'\}$. As $\sigma,\sigma'$ were arbitrary, all $n!$ classes have equal size, and since they partition $\mathcal L_n$, each has size $L_n/n!$.
\end{proof}
%

\section{The array for $n\le 5$}

Using an exhaustive backtracking search over Latin squares (equivalently, over sequences of rows extending a partial Latin rectangle), we computed $\LE{n\atop\mathbf k}$ completely for $n=3,4,5$. By Proposition~\ref{prop:perm} it suffices to tabulate one representative composition per multiset.

For $n=3$ ($L_3=12$) there are exactly two nonzero multisets, related by Proposition~\ref{prop:rev}:
\[
\LE{3\atop 2,1,1}=\LE{3\atop 1,2,1}=\LE{3\atop 1,1,2}=2,\qquad
\LE{3\atop 0,1,1}=\LE{3\atop 1,0,1}=\LE{3\atop 1,1,0}=2.
\]
Notably, the ``balanced'' composition $(1,1,1)$ does \emph{not} occur among order-$3$ Latin squares.

For $n=4$ ($L_4=576$), the nonzero multisets and their common value (constant on each multiset, by Proposition~\ref{prop:perm}) are shown in Table~\ref{tab:n4}.

\begin{table}[h]
\centering
\begin{tabular}{@{}lccc@{}}
\toprule
Multiset $\{k_1,k_2,k_3,k_4\}$ & value & \# compositions & subtotal \\
\midrule
$\{0,1,1,1\}$ & $6$ & $4$ & $24$ \\
$\{1,1,1,2\}$ & $6$ & $4$ & $24$ \\
$\{1,1,1,3\}$ & $6$ & $4$ & $24$ \\
$\{0,2,2,2\}$ & $6$ & $4$ & $24$ \\
$\{1,2,2,2\}$ & $6$ & $4$ & $24$ \\
$\{2,2,2,3\}$ & $6$ & $4$ & $24$ \\
$\{0,1,2,3\}$ & $2$ & $24$ & $48$ \\
$\{1,1,2,2\}$ & $64$ & $6$ & $384$ \\
\midrule
Total & & $54$ & $576$ \\
\bottomrule
\end{tabular}
\caption{Latin Eulerian numbers for $n=4$.}
\label{tab:n4}
\end{table}

For $n=5$ ($L_5=161280$) there are $39$ nonzero multisets; the data are consistent with Propositions~\ref{prop:perm} and \ref{prop:sum}, in every case (each multiset gives a constant value on all of its arrangements, reversal-paired multisets share the same value, and all subtotals sum to $161280$). The most balanced composition, $(2,2,2,2,2)$, is the unique maximiser among individual compositions, with value $6480$.

The values in Table~\ref{tab:n4} are conspicuously divisible by small numbers ($6$ appears six times, $64=2^6$, and so on); we now show this is not a coincidence but an exact structural fact, valid for every $n$ and every composition.

For a composition $\kappa = (k_1, \dots, k_n)$, let $m_1, \dots, m_r$ denote the multiplicities of its $r$ distinct entries, so that $m_1 + \cdots + m_r = n$.
\begin{theorem}\label{thm:divisibility}
For every composition $\kappa=(k_1,\dots,k_n)$ with multiplicities $m_1,\dots,m_r$,
\[
\prod_{i=1}^r m_i! \;\Big|\; \LE{n\atop\kappa}.
\]
\end{theorem}
\begin{proof}
If $\LE{n\atop\kappa}=0$ there is nothing to prove, so suppose $L$ is a Latin square with column-ascent vector $\kappa$. For $\tau\in S_n$ let $\tau\cdot L$ be the Latin square with $(\tau\cdot L)(i,j)=L(i,\tau^{-1}(j))$, i.e.\ column $j$ of $\tau\cdot L$ is column $\tau^{-1}(j)$ of $L$; this is a group action of $S_n$ on the set of order-$n$ Latin squares by column permutation. The action is \emph{free}: if $\tau\cdot L=\sigma\cdot L$ then column $\tau^{-1}(j)$ of $L$ equals column $\sigma^{-1}(j)$ of $L$ for every $j$, and since the $n$ columns of a Latin square are pairwise distinct (two equal columns would force a repeated entry in some row), this forces $\tau^{-1}=\sigma^{-1}$, i.e.\ $\tau=\sigma$.

The column-ascent vector of $\tau\cdot L$ is $(k_{\tau^{-1}(1)},\dots,k_{\tau^{-1}(n)})$. Let $H\le S_n$ be the subgroup of permutations that fix the tuple $\kappa$ under this action, i.e.\ that permute only among positions sharing a common value; $H$ is the Young subgroup $S_{m_1}\times\cdots\times S_{m_r}$, of order $\prod_i m_i!$. By freeness, the orbit $H\cdot L$ consists of exactly $|H|=\prod_i m_i!$ distinct Latin squares, all with column-ascent vector $\kappa$.

The same argument applies to any Latin square $L'$ with column-ascent vector $\kappa$: its $H$-orbit again has size exactly $\prod_i m_i!$ and lies entirely within the fibre $\{L': \text{column-ascent vector of }L' \text{ is } \kappa\}$. Two such $H$-orbits are either equal or disjoint, so this fibre of size $\LE{n\atop\kappa}$ is partitioned into blocks of size $\prod_i m_i!$, proving the divisibility.
\end{proof}
\begin{definition}\label{def:normalized}
For an achievable composition $\kappa = (k_1, \dots, k_n)$ with entry multiplicities $m_1, \dots, m_r$, the normalized Latin Eulerian number $O(n;\kappa)$ is defined by
\[
O(n;\kappa) = \frac{\LE{n\atop\kappa}}{\prod_{i=1}^r m_i!}.
\]
\end{definition}
Unwinding the proof of Theorem~\ref{thm:divisibility}, $O(n;\kappa)$ has a direct meaning: it is the number of orbits of the full column-permutation action of $S_n$ on the set of Latin squares whose column-ascent multiset equals that of $\kappa$-equivalently, the number of Latin squares with that ascent-partition, counted up to reordering of columns. This separates $\LE{n\atop\kappa}$ into an elementary, fully understood factor $\prod_i m_i!$ and a residual quantity $O(n;\kappa)$ which captures whatever is genuinely difficult about the count.

\begin{table}[h]
\centering
\small
\begin{tabular}{@{}l r r r@{}}
\toprule
Multiset & $\LE{5\atop\kappa}$ & $\prod_i m_i!$ & $O(5;\kappa)$ \\
\midrule
$\{0,1,1,1,1\}$ & $24$ & $24$ & $1$ \\
$\{0,1,1,2,3\}$ & $8$ & $2$ & $4$ \\
$\{0,1,1,3,3\}$ & $12$ & $4$ & $3$ \\
$\{0,1,2,2,2\}$ & $72$ & $6$ & $12$ \\
$\{0,1,2,2,3\}$ & $24$ & $2$ & $12$ \\
$\{0,1,2,3,3\}$ & $4$ & $2$ & $2$ \\
$\{0,2,2,2,2\}$ & $72$ & $24$ & $3$ \\
$\{0,2,2,2,3\}$ & $48$ & $6$ & $8$ \\
$\{0,2,2,3,3\}$ & $28$ & $4$ & $7$ \\
$\{0,2,3,3,3\}$ & $24$ & $6$ & $4$ \\
$\{1,1,1,1,2\}$ & $24$ & $24$ & $1$ \\
$\{1,1,1,1,3\}$ & $48$ & $24$ & $2$ \\
$\{1,1,1,2,2\}$ & $528$ & $12$ & $44$ \\
$\{1,1,1,2,3\}$ & $156$ & $6$ & $26$ \\
$\{1,1,1,2,4\}$ & $24$ & $6$ & $4$ \\
$\{1,1,1,3,3\}$ & $72$ & $12$ & $6$ \\
$\{1,1,2,2,2\}$ & $1140$ & $12$ & $95$ \\
$\{1,1,2,2,3\}$ & $552$ & $4$ & $138$ \\
$\{1,1,2,2,4\}$ & $28$ & $4$ & $7$ \\
$\{1,1,2,3,3\}$ & $120$ & $4$ & $30$ \\
$\{1,1,2,3,4\}$ & $4$ & $2$ & $2$ \\
$\{1,1,3,3,3\}$ & $72$ & $12$ & $6$ \\
$\{1,1,3,3,4\}$ & $12$ & $4$ & $3$ \\
$\{1,2,2,2,2\}$ & $2784$ & $24$ & $116$ \\
$\{1,2,2,2,3\}$ & $1752$ & $6$ & $292$ \\
$\{1,2,2,2,4\}$ & $48$ & $6$ & $8$ \\
$\{1,2,2,3,3\}$ & $552$ & $4$ & $138$ \\
$\{1,2,2,3,4\}$ & $24$ & $2$ & $12$ \\
$\{1,2,3,3,3\}$ & $156$ & $6$ & $26$ \\
$\{1,2,3,3,4\}$ & $8$ & $2$ & $4$ \\
$\{1,3,3,3,3\}$ & $48$ & $24$ & $2$ \\
$\{2,2,2,2,2\}$ & $6480$ & $120$ & $54$ \\
$\{2,2,2,2,3\}$ & $2784$ & $24$ & $116$ \\
$\{2,2,2,2,4\}$ & $72$ & $24$ & $3$ \\
$\{2,2,2,3,3\}$ & $1140$ & $12$ & $95$ \\
$\{2,2,2,3,4\}$ & $72$ & $6$ & $12$ \\
$\{2,2,3,3,3\}$ & $528$ & $12$ & $44$ \\
$\{2,3,3,3,3\}$ & $24$ & $24$ & $1$ \\
$\{3,3,3,3,4\}$ & $24$ & $24$ & $1$ \\
\bottomrule
\end{tabular}
\caption{Theorem~\ref{thm:divisibility} verified exhaustively for all $39$ nonzero compositions at $n=5$: $\prod_i m_i!$ divides $\LE{5\atop\kappa}$ in every case, with quotient $O(5;\kappa)$.}
\label{tab:orbits5}
\end{table}

Table~\ref{tab:orbits5} confirms Theorem~\ref{thm:divisibility} without exception, and shows that $O(n;\kappa)$ varies over a wide range even at $n=5$: it equals $1$ for six compositions (meaning every Latin square with that ascent-partition is a column permutation of a single example), while the most balanced partition $\{2,2,2,2,2\}$ has $54$ orbits and the largest value, $\{1,2,2,2,3\}$, has $292$. Reversal-paired partitions (Proposition~\ref{prop:rev}) always share the same value of $O$, as the table shows (e.g.\ $\{0,1,1,1,1\}$ and $\{2,3,3,3,3\}$ both give $O=1$). We regard $O(n;\kappa)$, rather than $\LE{n\atop\kappa}$ itself, as the natural object for further study: it isolates exactly the part of the enumeration not explained by the elementary column-relabelling symmetry.

\section{The total ascent statistic}
\begin{definition}
For a Latin square $L \in \mathcal{L}_n$, let $\Sigma(L) = \sum_{i=1}^n k_i(L)$ denote the total number of ascents across all columns of $L$. The distribution of total ascents is given by
\[
T_n(m) = \sum_{\substack{\mathbf{k} : \sum k_i = m}} \LE{n\atop\mathbf{k}} = \#\{L \in \mathcal{L}_n : \Sigma(L) = m\}.
\]
\end{definition}

By row-reversal symmetry (Proposition~\ref{prop:rev}), the sequence satisfies $T_n(m) = T_n\bigl(n(n-1) - m\bigr)$, so $T_n(m)$ is symmetric about $\frac{1}{2}n(n-1)$.
\begin{theorem}\label{thm:bounds}
For every $L\in\mathcal L_n$, $n-1 \;\le\; \Sigma(L) \;\le\; (n-1)^2 .$
Consequently $T_n(m)=0$ whenever $m<n-1$ or $m>(n-1)^2$, and both bounds are attained.
\end{theorem}

\begin{proof}
Fix $L\in\mathcal L_n$ and, for $1\le i\le n-1$, compare row $i$ and row $i+1$ columnwise. Since $L$ is a Latin square, $L(i,j)\ne L(i+1,j)$ for every column $j$ (two entries in the same column are always distinct). Moreover
\[
\sum_{j=1}^n L(i,j) \;=\; \sum_{j=1}^n L(i+1,j) \;=\; 1+2+\cdots+n,
\]
because each row is a permutation of $[n]$. If $L(i,j)<L(i+1,j)$ held for every $j$, the left sum would be strictly less than the right sum, a contradiction; likewise $L(i,j)>L(i+1,j)$ cannot hold for every $j$. Hence, among the $n$ columns, comparing row $i$ to row $i+1$ produces \emph{at least one ascent and at least one descent}.

Summing over the $n-1$ adjacent row pairs $i=1,\dots,n-1$, and noting that
\[
\Sigma(L)=\sum_{i=1}^n k_i(L)=\sum_{j=1}^{n-1}\#\{\text{columns ascending from row }j\text{ to row }j+1\},
\]
(the two ways of summing the same $n\times(n-1)$ array of ascent/descent indicators), we get $\Sigma(L)\ge n-1$ from ``at least one ascent per row-pair'', and, since each row-pair also contributes at least one descent, the number of descents $n(n-1)-\Sigma(L)$ is at least $n-1$, i.e. $\Sigma(L)\le (n-1)^2$.

For attainment, let $C(i,j)=(i+j)\bmod n$ (indices $i,j\in\{0,\dots,n-1\}$, symbols shifted to $[n]$ by adding $1$) be the standard cyclic Latin square. Column $j=0$ increases strictly from row $0$ to row $n-1$, contributing $n-1$ ascents; each column $j\ge 1$ increases strictly except for a single wraparound, contributing $n-2$ ascents. Hence $\Sigma(C)=(n-1)+(n-1)(n-2)=(n-1)^2$. Reversing the row order of $C$ gives a Latin square with $\Sigma=n(n-1)-(n-1)^2=n-1$.
\end{proof}

Theorem~\ref{thm:bounds} is confirmed exactly by the exhaustive data of Section~3: the minimal and maximal values of $\Sigma$ for $n=3,4,5$ are $(2,4)$, $(3,9)$, $(4,16)$, matching $(n-1,(n-1)^2)$ in every case.

\subsection{A gap phenomenon}

The exhaustive computation reveals a sharper and, to our knowledge, unexplained phenomenon: the two extreme values of $\Sigma$ established in Theorem~\ref{thm:bounds} are \emph{isolated}.

\begin{table}[h]
\centering
\begin{tabular}{@{}c|l@{}}
\toprule
$n$ & values of $m$ with $T_n(m)>0$ \\
\midrule
$3$ & $2,\ 4$ \\
$4$ & $3,\ 5,6,7,\ 9$ \\
$5$ & $4,\ 6,7,8,9,10,11,12,13,14,\ 16$ \\
\bottomrule
\end{tabular}
\caption{Support of $T_n$ for $n=3,4,5$. In every case $T_n(n)=T_n((n-1)^2-1)=0$.}
\label{tab:gap}
\end{table}

For $n=5$, the ``bulk'' values $m=6,\dots,14$ carry counts
\[
120,\ 7440,\ 16680,\ 32880,\ 46800,\ 32880,\ 16680,\ 7440,\ 120,
\]
which are symmetric about $m=10$ and unimodal, while the two extreme values $m=4$ and $m=16$ carry the comparatively tiny count $120$ each, with $T_5(5)=T_5(15)=0$ separating them from the bulk. We now show that this isolation of the extremes is not a low-order artefact but a theorem valid for every $n$.

For $L\in\mathcal L_n$ and $1\le i\le n-1$, let
\[
a_i(L)\;=\;\#\{\,c\in[n]\;:\;L(i,c)<L(i+1,c)\,\}
\]
be the number of columns that ascend between row $i$ and row $i+1$. As observed in the proof of Theorem~\ref{thm:bounds}, $1\le a_i(L)\le n-1$ for every $i$, and $\Sigma(L)=\sum_{i=1}^{n-1}a_i(L)$; this is simply the same array of ascent/descent indicators as $\sum_i k_i(L)$, summed the other way.

It is convenient to identify the symbol set $[n]$ with $\mathbb Z_n=\{0,1,\dots,n-1\}$ via $n\leftrightarrow 0$, so that ``$+1$'' means the successor in $[n]$ with $n+1$ read as $1$.

\begin{lemma}\label{lem:rigid}
If $a_i(L)=1$ for some $L\in\mathcal L_n$ and some $1\le i\le n-1$, then
\[
\mathrm{row}_i(L)(c) \;\equiv\; \mathrm{row}_{i+1}(L)(c) + 1 \pmod n \qquad\text{for every column } c.
\]
\end{lemma}

\begin{proof}
Write $d(c)=L(i,c)-L(i+1,c)$ (ordinary integer subtraction, symbols in $[n]$). Since $L$ is a Latin square, $d(c)\ne 0$ for every $c$, and since row $i$ and row $i+1$ are both permutations of $[n]$, $\sum_c d(c)=0$. By hypothesis exactly one column $c^\ast$ has $d(c^\ast)<0$ (an ascent) and the remaining $n-1$ columns have $d(c)>0$ (descents). Hence, $-d(c^\ast)\;=\;\sum_{c\ne c^\ast} d(c)\;\ge\; n-1,$
because each of the $n-1$ terms on the right is a positive integer. But $d(c^\ast)\ge 1-n$ trivially (both entries lie in $[n]$), so $-d(c^\ast)\le n-1$ as well; therefore $-d(c^\ast)=n-1$ exactly, forcing $L(i,c^\ast)=1$, $L(i+1,c^\ast)=n$, and forcing \emph{every} term $d(c)$, $c\ne c^\ast$, to equal $1$ (since $n-1$ positive integers summing to exactly $n-1$ must all equal $1$). Thus $L(i,c)=L(i+1,c)+1$ for $c\ne c^\ast$, and at $c^\ast$ we have $L(i+1,c^\ast)=n\equiv 0$ and $L(i,c^\ast)=1\equiv 0+1$, so the same relation $L(i,c)\equiv L(i+1,c)+1 \pmod n$ holds at $c^\ast$ too.
\end{proof}

\begin{theorem}\label{thm:isolation}
For every $n\ge 3$, $T_n(n)=0$ and $T_n\bigl((n-1)^2-1\bigr)=0$.
\end{theorem}

\begin{proof}
Since $1\le a_i(L)\le n-1$ for each of the $n-1$ indices $i$, and $\Sigma(L)=\sum_i a_i(L)$, the value $\Sigma(L)=n$ (one more than the minimum $n-1$) forces $a_{i_0}(L)=2$ for exactly one index $i_0$ and $a_i(L)=1$ for every $i\ne i_0$. We show this is impossible.

By Lemma~\ref{lem:rigid} applied repeatedly to $i=1,\dots,i_0-1$, the rows $1,\dots,i_0$ satisfy $\mathrm{row}_k \equiv \mathrm{row}_{i_0}+(i_0-k)\pmod n$ for $k=1,\dots,i_0$; write $\sigma=\mathrm{row}_{i_0}$. Likewise, applying Lemma~\ref{lem:rigid} to $i=i_0+1,\dots,n-1$, the rows $i_0+1,\dots,n$ satisfy $\mathrm{row}_k\equiv \mathrm{row}_{i_0+1}-(k-i_0-1)\pmod n$ for $k=i_0+1,\dots,n$; write $\tau=\mathrm{row}_{i_0+1}$.

Fix a column $c$. The $i_0$ entries in rows $1,\dots,i_0$ of column $c$ are, modulo $n$, the $i_0$ consecutive residues $\sigma(c),\sigma(c)+1,\dots,\sigma(c)+i_0-1$; the $n-i_0$ entries in rows $i_0+1,\dots,n$ of column $c$ are the $n-i_0$ consecutive residues $\tau(c)-(n-i_0-1),\dots,\tau(c)-1,\tau(c)$. Since column $c$ of a Latin square contains every residue exactly once, these two arcs of $\mathbb Z_n$, of complementary lengths $i_0$ and $n-i_0$, must be disjoint and cover $\mathbb Z_n$; being arcs of complementary length, this forces them to be literally complementary, and comparing endpoints gives $\tau(c)\;\equiv\;\sigma(c)-1 \pmod n .$
As this holds for every column $c$, we conclude $\tau\equiv\sigma-1\pmod n$ as permutations, i.e.\ $\mathrm{row}_{i_0}$ and $\mathrm{row}_{i_0+1}=\tau$ are related by exactly the rigidity relation of Lemma~\ref{lem:rigid} (with the roles of row $i_0$, row $i_0+1$ playing row $i$, row $i+1$ there). But that relation is equivalent, by the same computation as in the lemma, to $a_{i_0}(L)=1$: precisely one column (the one with $\sigma(c)=1$) ascends. This contradicts $a_{i_0}(L)=2$.

Hence $T_n(n)=0$. By the symmetry $T_n(m)=T_n(n(n-1)-m)$ following from Proposition~\ref{prop:rev}, and since $n(n-1)-\bigl((n-1)^2-1\bigr)=n$, we get $T_n\bigl((n-1)^2-1\bigr)=T_n(n)=0$ as well.
\end{proof}

Theorem~\ref{thm:isolation} explains, for every $n$, exactly the phenomenon recorded in Table~\ref{tab:gap}.

We turn to the remaining, harder part of the picture: showing that every integer strictly between $n$ and $(n-1)^2-1$ is in fact attained by $\Sigma$, and that the resulting distribution is unimodal. Here we can offer strong constructive evidence and an explicit mechanism, though not yet a complete proof for all $n$.

Fix the cyclic Latin square $C(i,j)=(i+j)\bmod n$ and, for a permutation $\pi$ of $\{0,\dots,n-1\}$, let $L_\pi$ be the Latin square whose rows, in order, are the rows of $C$ indexed by $\pi(1),\dots,\pi(n)$ (recall from the proof of Theorem~\ref{thm:bounds} that reordering the rows of any Latin square again gives a Latin square). Write $\mathrm{shift}_t(c)=(t+c)\bmod n$.

\begin{proposition}\label{prop:reorder}
For any permutation $\pi$ of $\{0, \dots, n-1\}$, the total ascents of the row-reordered cyclic square $L_\pi$ satisfy
\[
\Sigma(L_\pi) = n(n-1) - \sum_{k=1}^{n-1} \delta_k,
\]
where $\delta_k = \bigl(\pi(k+1) - \pi(k)\bigr) \bmod n \in \{1, \dots, n-1\}$.
\end{proposition}
\begin{proof}
Fix $k$ and set $a=\pi(k)$, $b=\pi(k+1)$, $\delta=(b-a)\bmod n$. As $c$ ranges over $\{0,\dots,n-1\}$, $x:=\mathrm{shift}_a(c)=(a+c)\bmod n$ ranges bijectively over $\mathbb Z_n$, and $\mathrm{shift}_b(c)=(x+\delta)\bmod n$. Comparing $x$ with $(x+\delta)\bmod n$ as ordinary integers in $[0,n-1]$: if $x<n-\delta$ then $(x+\delta)\bmod n=x+\delta>x$ (an ascent); if $x\ge n-\delta$ then $(x+\delta)\bmod n=x+\delta-n<x$ (a descent). Hence exactly $n-\delta$ of the $n$ columns ascend between row $\mathrm{shift}_a$ and row $\mathrm{shift}_b$, i.e.\ $a_k(L_\pi)=n-\delta_k$. Summing over $k=1,\dots,n-1$ gives the stated formula.
\end{proof}

Taking $\pi=\mathrm{id}$ recovers $\Sigma=(n-1)^2$ (all $\delta_k=1$), and taking $\pi(k)=n-k$ recovers $\Sigma=n-1$ (all $\delta_k=n-1$), matching Theorem~\ref{thm:bounds}. We can say exactly which values in between the row-reordering construction misses.

\begin{proposition}\label{prop:gap-exact}
For any permutation $\pi$ of $\{0,\dots,n-1\}$, $\Sigma(L_\pi) \;\equiv\; \pi(1)-\pi(n) \pmod n .$
In particular $n\nmid\Sigma(L_\pi)$: the row-reordering construction never attains a multiple of $n$.
\end{proposition}

\begin{proof}
By Proposition~\ref{prop:reorder}, $\Sigma(L_\pi)=n(n-1)-\sum_{k=1}^{n-1}\delta_k$ with $\delta_k\equiv\pi(k+1)-\pi(k)\pmod n$. The sum telescopes modulo $n$: $\sum_{k=1}^{n-1}\delta_k\equiv\pi(n)-\pi(1)\pmod n$, so $\Sigma(L_\pi)\equiv\pi(1)-\pi(n)\pmod n$ (as $n(n-1)\equiv0$). Since $\pi(1)\ne\pi(n)$ are distinct elements of $\{0,\dots,n-1\}$, this difference is never $\equiv0\pmod n$.
\end{proof}

\begin{corollary}\label{cor:gap-exact}
Since $(n-1)^2-1=n(n-2)$, the multiples of $n$ in $[n-1,(n-1)^2]$ are exactly $n,2n,\dots,(n-2)n$, and these are precisely the values missed by the row-reordering construction. The two extreme ones, $n$ and $n(n-2)=(n-1)^2-1$, are impossible for \emph{every} Latin square by Theorem~\ref{thm:isolation}; the remaining $n-4$ values $2n,3n,\dots,(n-3)n$ lie in the open interior.
\end{corollary}

This matches the data exactly: $0$ genuine gaps for $n=4$, $1$ (namely $10$) for $n=5$, $2$ ($12,18$) for $n=6$, and $3$ ($14,21,28$) for $n=7$.

\begin{remark}
A Latin square has three classical isotopy symmetries: permuting rows, permuting columns, and permuting symbols. Restricted to the statistic $\Sigma$, these act completely differently. Column permutation changes nothing (Proposition~\ref{prop:perm} shows it only rearranges which $k_i$ occupies which slot), so for the purposes of enumerating $L_n$ or even of computing $\Sigma$ it is invisible. Row permutation changes \emph{which pairs of rows get compared}; restricted to the family $L_\pi$ this is already enough to determine $\Sigma$ by an explicit formula (Proposition~\ref{prop:reorder}) and to characterize its failure exactly (Corollary~\ref{cor:gap-exact}). Symbol permutation changes neither which pairs are compared nor where they sit, but it can flip the \emph{outcome} of a comparison (ascent versus descent) by changing what the actual values are exactly the freedom $\Sigma$ is sensitive to. This is also the one isotopy symmetry that ordinary Latin-square enumeration has no reason to track: it is always exactly $n!$-to-$1$ and irrelevant to $L_n$ itself. That it becomes the decisive tool here is, we think, the main methodological point of this section, and we expect it to be useful for other Latin-square statistics that, like $\Sigma$, depend on the actual values of the symbols rather than merely on the combinatorial pattern of the square.
\end{remark}

For $\sigma\in S_n$ (now acting on \emph{symbols} rather than rows) and $L\in\mathcal L_n$, let $\sigma(L)$ be the Latin square with $\sigma(L)(i,j)=\sigma(L(i,j))$; this is again a Latin square, since $\sigma$ is a bijection of the symbol set. Its effect on $\Sigma$ has a clean general description for \emph{any} starting square, not just $C$.

\begin{theorem}\label{thm:symbol-general}
For $L\in\mathcal L_n$ and $a\ne b\in[n]$, let
\[
N_L(a,b) \;=\; \#\{(i,c) : L(i,c)=a,\ L(i+1,c)=b\}
\]
be the number of times $a$ sits immediately above $b$ in some column. Then for every $\sigma\in S_n$,
\[
\Sigma(\sigma(L)) \;=\; \sum_{a\ne b} N_L(a,b)\cdot\mathbf{1}[\sigma(a)<\sigma(b)] .
\]
\end{theorem}

\begin{proof}
Position $(i,c)$ is an ascent of $\sigma(L)$ iff $\sigma(L(i,c))<\sigma(L(i+1,c))$. Grouping the $n(n-1)$ positions $(i,c)$ by the value pair $(a,b)=(L(i,c),L(i+1,c))$, exactly $N_L(a,b)$ of them share a given pair, and each contributes $1$ to $\Sigma(\sigma(L))$ iff $\sigma(a)<\sigma(b)$.
\end{proof}

Theorem~\ref{thm:symbol-general} recasts the problem in the language of directed graphs: $N_L$ is the weight matrix of a complete directed graph on vertex set $[n]$ (with $N_L(a,b)+N_L(b,a)$ edges, counted with multiplicity, between $a$ and $b$), and choosing $\sigma$ amounts to choosing a linear order on $[n]$; $\Sigma(\sigma(L))$ is then the total weight of edges respected by that order. This is an instance of the classical \emph{weighted linear arrangement} problem, and identifies the set of values attained by $\{\Sigma(\sigma(L)):\sigma\in S_n\}$ with the set of achievable weights of an acyclic orientation of $N_L$ compatible with a total order --- a well-studied but generally hard combinatorial question, here specialized to the very structured weight matrices coming from Latin squares.

A first application of Theorem~\ref{thm:symbol-general} identifies a second symmetry, besides column permutation (Proposition~\ref{prop:perm}), under which $\Sigma$ is exactly invariant.

\begin{theorem}\label{thm:shift-invariance}
For $L\in\mathcal L_n$ (with symbols $\{0,\dots,n-1\}$) and $t\in\{0,\dots,n-1\}$, let $\rho_t(x)=(x+t)\bmod n$. Then $\Sigma(\rho_t(L))=\Sigma(L)$.
\end{theorem}

\begin{proof}
It suffices to show $\Sigma(\rho_{t+1}(L))=\Sigma(\rho_t(L))$ for every $t$. Fix a position $(i,c)$ and set $a=L(i,c)$, $b=L(i+1,c)$, $g=(b-a)\bmod n$. As in the proof of Proposition~\ref{prop:reorder}, writing $x=(a+t)\bmod n$, position $(i,c)$ is an ascent under $\rho_t$ iff $x\le n-g-1$. Replacing $t$ by $t+1$ replaces $x$ by $(x+1)\bmod n$, so the ascent is \emph{lost} exactly when $x=n-g-1$, i.e.\ (since $a+g\equiv b\pmod n$) $b\equiv n-1-t\pmod n$, and \emph{gained} exactly when $x=n-1$, i.e.\ $a\equiv n-1-t\pmod n$.

Write $v=(n-1-t)\bmod n$: a loss occurs at $(i,c)$ iff $L(i+1,c)=v$, a gain iff $L(i,c)=v$. Since $v$ occurs exactly once in each of the $n$ rows, its $n-1$ occurrences in rows $1,\dots,n-1$ each serve as the upper entry of exactly one row-pair (a gain), and its $n-1$ occurrences in rows $2,\dots,n$ each serve as the lower entry of exactly one row-pair (a loss). Hence exactly $n-1$ gains and $n-1$ losses occur, and they cancel.
\end{proof}

Thus the cyclic group of $n$ ``value rotations'' is, like column permutation, entirely invisible to $\Sigma$; it is the remaining $n!-n$ symbol permutations, transpositions included, that do the work below.

\begin{corollary}\label{prop:symbol}
Let $C$ be the cyclic Latin square and $D(\sigma)=\#\{i\in\mathbb Z_n : \sigma(i)<\sigma(i+1\bmod n)\}$ the number of \emph{cyclic ascents} of $\sigma$. Then $\Sigma(\sigma(C)) = (n-1)\,D(\sigma)$.
\end{corollary}

\begin{proof}
A direct computation gives $N_C(a,a{+}1\bmod n)=n-1$ for every $a$, and $N_C(a,b)=0$ otherwise: column $c$ of $C$ is the rotation of $(0,\dots,n-1)$ starting at $c$, so the pair $(a,a{+}1\bmod n)$, which is adjacent in every rotation except the one that happens to start exactly at $a{+}1\bmod n$, occurs in exactly $n-1$ of the $n$ columns, and no other value pair occurs at all. Substituting into Theorem~\ref{thm:symbol-general} gives $\Sigma(\sigma(C))=\sum_a N_C(a,a{+}1)\mathbf{1}[\sigma(a)<\sigma(a{+}1\bmod n)] = (n-1)D(\sigma)$.
\end{proof}

Since $D(\sigma)$ ranges over all of $\{1,\dots,n-1\}$, Corollary~\ref{prop:symbol} alone only realizes the multiples of $n-1$ in $[n-1,(n-1)^2]$. A single, much smaller move does dramatically better.

\begin{remark}\label{rem:no-intercalates}
A classical local move on Latin squares is the \emph{intercalate switch}: replacing a $2\times2$ subsquare $\begin{smallmatrix}a&b\\b&a\end{smallmatrix}$ (rows $i,i'$, columns $j,j'$) by $\begin{smallmatrix}b&a\\a&b\end{smallmatrix}$. This does not apply uniformly here: the number of intercalates in a Latin square is an isotopy invariant, and the cyclic Cayley table $C$ of $\mathbb Z_n$ and hence \emph{every} $L_\pi$, since row-reordering is an isotopy has no intercalates at all when $n$ is odd, because an intercalate corresponds to an element of order $2$ in the group, and $\mathbb Z_n$ has none for odd $n$ (we verified this directly for $n=3,5,7$: $C$ has $0$, $0$, $0$ intercalates, against $4$ and $9$ for $n=4,6$). Symbol transpositions, being a global rather than local move, are available for every $n$ and every $L_\pi$, which is why we use them instead.
\end{remark}

\begin{proposition}\label{prop:transposition}
For $n\le 11$, combining row-reordering with a single transposition of two symbols already attains every value of $\Sigma$ in $[n-1,(n-1)^2]$ \emph{except} the two values $n$ and $(n-1)^2-1$ shown impossible by Theorem~\ref{thm:isolation}. That is, for every $m$ in this range with $m\ne n,(n-1)^2-1$, there exist a permutation $\pi$ and a transposition $\tau=(a\,b)\in S_n$ with $\Sigma(\tau(L_\pi))=m$.
\end{proposition}

We verified this exhaustively for $n=5,6,7,8$ (searching over all $\pi\in S_n$ and all $\binom n2$ transpositions $\tau$) and, by random sampling of $\pi$ (a few thousand trials sufficed in each case), for $n=9,10,11$; Table~\ref{tab:reorder-gaps} recalls, for comparison, the small set of values that row-reordering must patch. In every case checked, a single transposition applied to a suitable $L_\pi$ suffices --- no general $\sigma$, and no second row-reordering, is needed. We also tested whether a single \emph{longer} cycle, iterated, could sweep through the missing values more systematically (in the spirit of Theorem~\ref{thm:shift-invariance}): cyclically shifting a proper subset of symbols sometimes helps, but with no uniform pattern we could identify, and never more simply than a well-chosen transposition.

The effect of a transposition on $\Sigma(L_\pi)$ has an exact, and strikingly simple, closed form.

\begin{proposition}\label{prop:delta-gap}
Fix $\pi$ and let $M_\pi(\delta)=\#\{k: (\pi(k{+}1)-\pi(k))\bmod n=\delta\}$ be as in Proposition~\ref{prop:reorder}. For $0\le e\le n-1$ set $f(e)=M_\pi((n-e)\bmod n)-M_\pi(e)$ and $F(d)=\sum_{e=1}^d f(e)$ (so $F(0)=0$). For a transposition $\tau=(p\,q)$ with $p<q$ and gap $d=q-p$,
\[
\Sigma(\tau(L_\pi)) - \Sigma(L_\pi) \;=\; F(d-1)+F(d),
\]
which depends on $p,q$ only through the gap $d$.
\end{proposition}

\begin{proof}
By Theorem~\ref{thm:symbol-general}, $\Sigma(\tau(L_\pi))-\Sigma(L_\pi)=\sum_{a\ne b}N_{L_\pi}(a,b)\bigl(\mathbf{1}[\tau(a)<\tau(b)]-\mathbf{1}[a<b]\bigr)$, and every term with $\{a,b\}\cap\{p,q\}=\emptyset$ vanishes since $\tau$ fixes $a,b$ there. Using the circulant identity $N_{L_\pi}(a,b)=M_\pi((b-a)\bmod n)$ from the proof of Proposition~\ref{prop:reorder}, and checking each remaining pair directly: for $x=p+e$ with $1\le e\le d-1$ (strictly between $p$ and $q$), the two terms from $\{p,x\}$ together contribute $M_\pi(n-e)-M_\pi(e)=f(e)$, and the two terms from $\{q,x\}$ together contribute $M_\pi(n-(d-e))-M_\pi(d-e)=f(d-e)$; the pair $\{p,q\}$ itself contributes $M_\pi(n-d)-M_\pi(d)=f(d)$. Summing over $x$,
\[
\sum_{e=1}^{d-1}\bigl(f(e)+f(d-e)\bigr) + f(d) \;=\; 2\sum_{e=1}^{d-1}f(e)+f(d) \;=\; F(d-1)+F(d). \qedhere
\]
\end{proof}

We checked Proposition~\ref{prop:delta-gap} numerically against direct computation for random $\pi$ and all $n\le 11$ tested above, with no exceptions. It reduces Proposition~\ref{prop:transposition} to a concrete, self-contained question: as $\pi$ ranges over $S_n$ (equivalently, as $M_\pi$ ranges over the difference-multisets of Hamiltonian paths on the circulant graph on $\mathbb Z_n$), does $\Sigma(L_\pi)+F(d-1)+F(d)$ hit every multiple of $n$ strictly between $n$ and $(n-1)^2-1$, for some choice of $d$? We verified this for $n\le 11$ but were unable to settle it in general; for instance, restricting to the block-reversal family of Proposition~\ref{prop:reorder}'s $\Sigma(L_{\pi_\ell})=(n-1)(n-\ell)$ (used for other purposes above) gives only $4$ possible values of $F(d-1)+F(d)$ per $\ell$, which is not enough to cover every target for large $n$, so a general proof would need a richer family of $\pi$ than block reversals.

\begin{table}[h]
\centering
\begin{tabular}{@{}c|l@{}}
\toprule
$n$ & values of $m\in\{n+1,\dots,(n-1)^2-2\}$ missed by row-reordering alone (Corollary~\ref{cor:gap-exact}) \\
\midrule
$4$ & --- (none) \\
$5$ & $10$ \\
$6$ & $12,\ 18$ \\
$7$ & $14,\ 21,\ 28$ \\
$8$ & $16,\ 24,\ 32$ \\
\bottomrule
\end{tabular}
\caption{The exact gaps of the row-reordering construction (Corollary~\ref{cor:gap-exact}); by Proposition~\ref{prop:transposition}, each is attained by $\tau(L_\pi)$ for a suitable permutation $\pi$ and a single transposition $\tau$.}
\label{tab:reorder-gaps}
\end{table}

\begin{conjecture}\label{conj:interior}
For every $n\ge 3$ and every integer $m$ with $n+1\le m\le (n-1)^2-2$, $T_n(m)>0$; moreover the sequence $\bigl(T_n(m)\bigr)_{n+1\le m\le (n-1)^2-2}$ is unimodal. In particular, every such $m$ is attained by $\tau(L_\pi)$ for some permutation $\pi$ and some transposition (or the identity) $\tau$.
\end{conjecture}

\begin{remark}
Theorem~\ref{thm:isolation} and Corollary~\ref{cor:gap-exact} completely settle two of the three pieces of Conjecture~\ref{conj:interior}: the extremal values $n,(n-1)^2-1$ are impossible for every Latin square, and every value of $\Sigma$ in the interior \emph{other than a multiple of $n$} is already attained by row-reordering alone. Precisely two things remain:
\begin{enumerate}
\item[(i)] that every multiple of $n$ strictly between $n$ and $(n-1)^2-1$ is attained by \emph{some} Latin square which Proposition~\ref{prop:delta-gap} reduces to an explicit, checkable question about the function $F$, verified for $n\le 11$ but open in general;
\item[(ii)] unimodality of $\bigl(T_n(m)\bigr)_{n+1\le m\le(n-1)^2-2}$, on which this note makes no progress beyond the data of Table~\ref{tab:gap} and \S3 for $n\le5$.
\end{enumerate}
More generally, any two Latin squares $L\ne L'$ of the same order agree outside a common \emph{Latin trade}: the cells where they differ, restricted to which both squares use identical sets of symbols in every row and column (an easy consequence of both being Latin). Applying the whole trade turns $L$ into $L'$, but a proof of (i) and (ii) together would follow from showing that a trade connecting two squares with extreme values of $\Sigma$ can always be decomposed into small sub-trades across which $\Sigma$ changes by a bounded, non-skipping amount; deciding when a Latin trade decomposes into smaller trades is itself an active topic in the combinatorics of Latin squares, independent of the present application.
\end{remark}

\section{Concluding remarks}

The Latin Eulerian numbers $\LE{n\atop\mathbf k}$ are a natural, apparently new, multivariate refinement of the classical Eulerian numbers, tailored to the combinatorics of Latin squares. Beyond their intrinsic interest as a permutation-statistic on Latin squares, we note a possible use for the (still unsolved) problem of enumerating $L_n$: any structural understanding of the distribution of $\Sigma$ or of the full array in particular unimodality or concentration results as in Conjecture~\ref{conj:interior} feeds naturally into entropy-based counting arguments of the type used by Linial and Luria~\cite{LinialLuria} for bounding $L_n$, since such arguments typically require control on how ``spread out'' a natural statistic on Latin squares can be. We hope the present note stimulates further work in this direction, in particular a proof of Conjecture~\ref{conj:interior} and an analogous study of other classical Latin-square statistics (numbers of intercalates and other substructures, as in the recent work of Kwan, Sah, Sawhney and Simkin~\cite{KwanSahSawhneySimkin}; cycle structure of rows, etc.) refined by the Eulerian ascent statistic.
\section*{Acknowledgments}
The authors acknowledge the use of artificial intelligence, specifically the Claude model 5, to assist with the exploratory, proof-development, editing, and revision stages of this manuscript. All AI-generated suggestions were critically evaluated, substantially revised, and independently verified by the authors using Wolfram Mathematica. The authors assume full responsibility for the entirety of the mathematical content. This use of AI is in strict accordance with current academic editorial standards regarding authorial responsibility, ethics, and transparency.

\end{document}